\documentclass[12pt,a4paper,twoside]{article}
\pdfoutput=1

\usepackage[applemac]{inputenc}
\usepackage[T1]{fontenc}
\usepackage[english]{babel}

\usepackage{fancyhdr}
\usepackage{indentfirst} 
\usepackage{newlfont}
\usepackage{verbatim}
\usepackage{enumerate}
\usepackage{hyperref}

\usepackage{amsthm}  
\usepackage{amssymb}  
\usepackage{amscd}
\usepackage{amsmath} 
\usepackage{latexsym} 
\usepackage{amsfonts}
\usepackage{amsbsy}   
\usepackage{mathrsfs} 
\usepackage{graphicx}
\usepackage{color}

\theoremstyle{plain}
\newtheorem{teo}{Theorem}
\newtheorem{prop}[teo]{Proposition}

\theoremstyle{definition}

\newtheorem{ese}[teo]{Example}

\theoremstyle{remark}
\newtheorem{oss}[teo]{Remark}

\newcommand{\rp}[1]{\ensuremath{\mathbb{RP}^{#1}}}
\newcommand{\s}[1]{\ensuremath{\mathbf{S}^{#1}}}

\begin{document}

\title{Lift in the $3$-sphere \\ of knots and links in lens spaces\footnote{Work supported by the University of Bologna, Department of Mathematics and Marco Polo funds for foreign research periods.}}

\author{Enrico Manfredi}

\maketitle

\begin{abstract}

An important geometric invariant of links in lens spaces is the lift in $\s3$ of a link $L \subset L(p,q)$, that is the counterimage $\widetilde L$ of $L$ under the universal covering of $L(p,q)$. 
If lens spaces are defined as a lens with suitable boundary identifications, then a link in $L(p,q)$ can be represented by a disk diagram, that is to say, a regular projection of the link on a disk. 
Starting from a disk diagram of $L$,
we obtain a diagram of the lift $\widetilde L$ in $\s3$. 
With this construction we are able to find different knots and links in $L(p,q)$ having equivalent lifts, that is to say, we cannot distinguish different links in lens spaces only from their lift. 
\\ {{\it Mathematics Subject
Classification 2010:} Primary 57M25, 57M27; Secondary 57M10.}\\

\end{abstract}

\begin{section}{Introduction}

The study of knots and links in the $3$-sphere is a widespread branch of mathematics. What happens for knots and links in other $3$-manifolds? Dehn surgery and mixed link diagrams are useful to represent any link $L$ in a $3$-manifold $M$, however there is not a good skein theory.
Different representations become really useful if we restrict to a particular class of closed $3$-manifolds, the lens spaces $L(p,q)$: several interesting results are shown in \cite{BG, BGH, Co, CM}.
Other results for $\rp3 \cong L(2,1)$ \cite{D, HL, G1, G2} and then for the general case \cite{CMM} are due to the particular representation on which we will focus.
Namely, if we consider the lens space $L(p,q)$ as the quotient of the unit ball $B^{3}$ where each boundary point is identified with the one in the opposite hemisphere after a planar reflection and a counterclockwise rotation of $2 \pi q /p$ radians around the polar axis, then we can project any link on the equatorial disk of $B^{3}$, obtaining a regular diagram for it, named disk diagram. 

In \cite{BG}, Baker and Grigsby consider a geometric invariant that could be really useful: given a link $L$ in $L(p,q)$, and assigned the cyclic covering map $P \colon \s3 \rightarrow L(p,q)$, the lift $\widetilde L$ of $L$ is the counterimage $P^{-1}(L)\subset \s3$. 
They produce a grid diagram for the lift but this representation cannot give much information about the properties of the invariant. For this reason we develop a geometric representation that, with the help of a link braid form, allows us to answer the following fundamental question: “Is the lift a complete invariant?”

The lift in $\s3$ of a link in $L(p,q)$ is exactly a $(p,q)$-lens link of Chbili \cite{C4}, and hence a freely periodic link in the 3-sphere \cite{H}. Our question can be re-phrased: “Are there links in $\s3$ that are freely periodic with respect to two different $(p,q)$-periodic transformations?”

For unoriented links up to ambient isotopy, the answer is negative: the lift is not a complete invariant. We construct several counterexamples, consisting of:
\begin{description} 
\item[\textmd{a)}] two non-equivalent knots in $L(p,\frac{p \pm 1}{2})$, $p>3$ and odd, with different homology class that are lifted both to the unknot;
\item[\textmd{b)}] a knot and a 2-component link in $L(4,1)$ that are lifted to the Hopf link;
\item[\textmd{c)}] an infinite family of cablings of b) that still have equivalent lift; the pairs of links may have a different number of components; in some cases they have the same number of components and the same homology class, we then find an example in which the pair has different Alexander polynomials.
\end{description}

Another important advantage of a diagram for the lift is a method to compute the fundamental quandle of links in lens spaces. The fundamental quandle of a link in a $3$-manifold is a geometric invariant that can be explicitly computed on a diagram only for links in $\s3$ \cite{J,M} and in $\rp3$ \cite{G2}.
Since the fundamental quandle of $L \subset L(p,q)$ is isomorphic to the fundamental quandle of its lift $\widetilde L$ \cite{M,FR}, we are able to compute it on the lift diagram. For the same reason, we know that the fundamental quandle cannot classify knots/links in lens spaces.

The paper is organized as follows.
In Section 2 we explain how to get a classical diagram in $\s3$ of the lift starting from the diagram of $L\subset L(p,q)$ defined in \cite{CMM}, and we show the connection with $(p,q)$-lens links. 
In Section 3 we show some interesting examples for split links, composite knots, cable links, then we develop a braid form that describes a subclass of links in lens spaces. In Section 4, exploiting this link braid form, we are able to find the examples a), b) and c), consisting of different links with equivalent lift, that is to say, the lift is not a complete invariant. At last, the case of oriented and diffeomorphic links is taken into account.

\end{section}


\begin{section}{Lift of links in lens spaces}

The results stated in this paper hold both in the \emph{Diff} category and in the \emph{PL} category, as well as in the \emph{Top} category if we consider only tame links.
In this section we recall the notion of disk diagram for a link in a lens space developed in \cite{GM} and \cite{CMM}, then we show how to get a planar diagram for the lift in $\s3$ of links in lens space.

\subsection{Two models for lens spaces}\label{sub1}
Let $p$ and $q$ be two coprime integers such that $ 0 \leqslant q < p$. The unit ball is the set $B^{3} = \{(x_{1},x_{2},x_{3}) \in \mathbb{R}^{3} \ | \ x_{1}^{2}+x_{2}^{2}+x_{3}^{2}\leqslant1\}$ and $E_{+}$ and $E_{-}$ are respectively the upper and the lower closed hemisphere of $\partial B^{3}$. The equatorial disk $B^{2}_{0}$ is defined by the intersection of the plane $x_{3}=0$ with $B^{3}$. Label with $N$ and $S$ respectively the "north pole" $(0,0,1)$ and the "south pole" $(0,0,-1)$ of $B^{3}$.
Let \mbox{$g_{p,q} \colon E_{+} \rightarrow E_{+}$} be the counterclockwise rotation of $2 \pi q /p$ radians around the $x_{3}$-axis, as in Figure~\ref{L(p,q)}, and let \hbox{$f_{3} \colon E_{+} \rightarrow E_{-}$} be the reflection with respect to the plane $x_{3}=0$.

\begin{figure}[h!]                      
\begin{center}                         
\includegraphics[width=8.4cm]{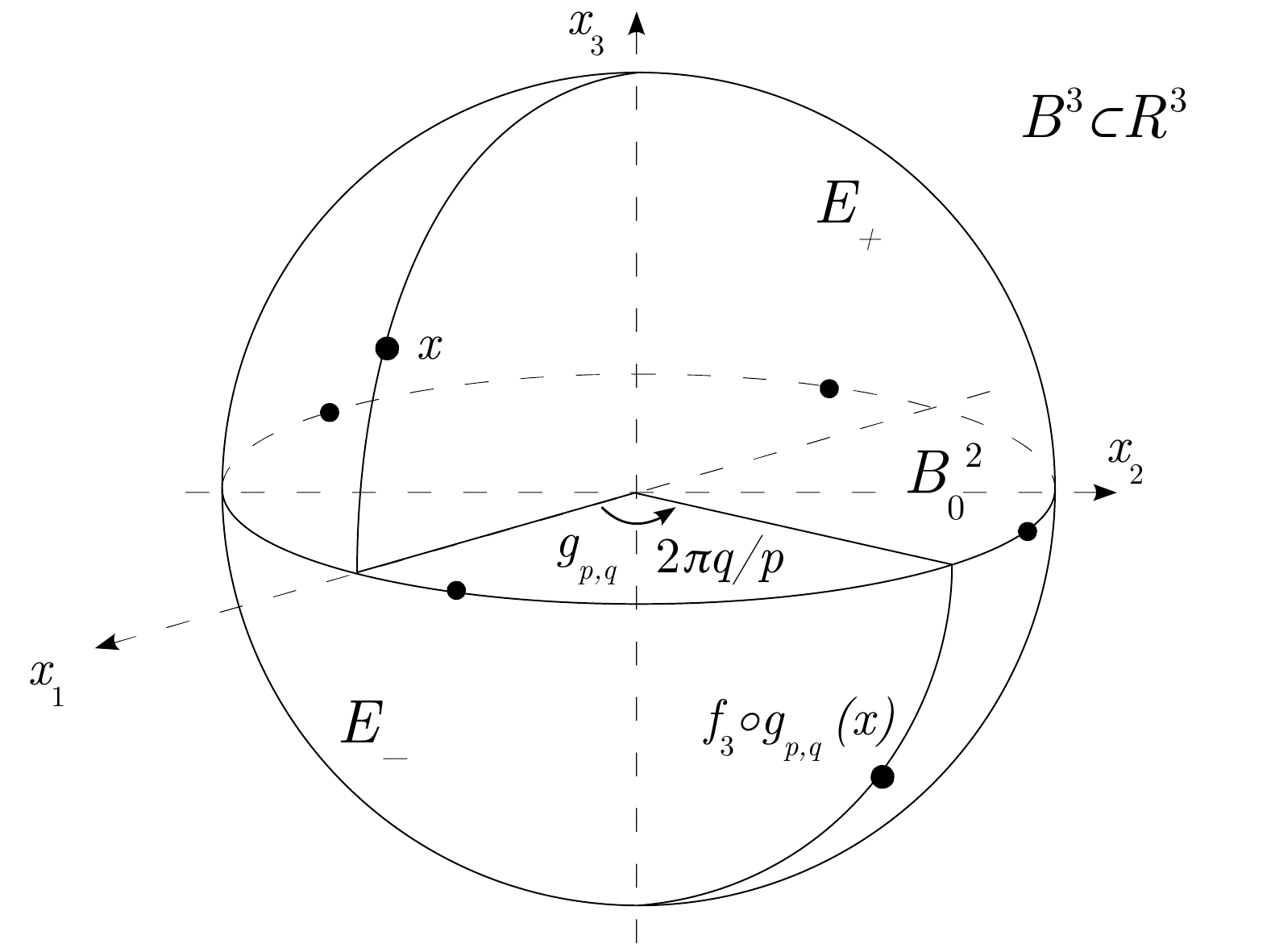}
\caption[legenda elenco figure]{Representation of $L(p,q)$.}\label{L(p,q)}
\end{center}
\end{figure}

The \emph{lens space} $L(p,q)$ is the quotient of $B^{3}$ by the equivalence relation on $\partial B^{3}$ which identifies $x \in E_{+}$ with $f_{3} \circ g_{p,q} (x) \in E_{-}$. The quotient map is denoted by \mbox{$F \colon B^{3} \rightarrow L(p,q)=B^{3} / \sim$}. Note that on the equator \mbox{$\partial B^{2}_{0}=E_{+} \cap E_{-}$} each equivalence class contains $p$ points, instead of the two points contained in equivalence classes outside the equator. The first example is $L(1,0)\cong \s{3}$ and the second example is $L(2,1) \cong \rp{3}$, where the construction gives the usual model of the projective space: opposite points on $\partial B^3$ are identified.


Another classical model for the lens space is the following:
consider $\s{3}$ as the join of two copies of $\s1$ (in a Hopf link configuration), put on it the action corresponding to the rotation of $2 \pi /p$ radians of the first circle and of $2 \pi q /p$ radians of the second one, according to Figure \ref{torusmodel}. Denote with $G_{p,q}$ the cyclic group generated by this action. Clearly $G_{p,q}$ is isomorphic to $ \mathbb{Z}_{p}$ and it acts without any fixed point, in a properly discontinuous way on $\s{3}$. Therefore the quotient space of $\s{3}$ is a $3$-manifold that indeed results to be the lens space $L(p,q)$. Denote with $P \colon \s3 \rightarrow L(p,q)$ the quotient map.

\begin{figure}[h!]                      
\begin{center}                         
\includegraphics[width=11.5cm]{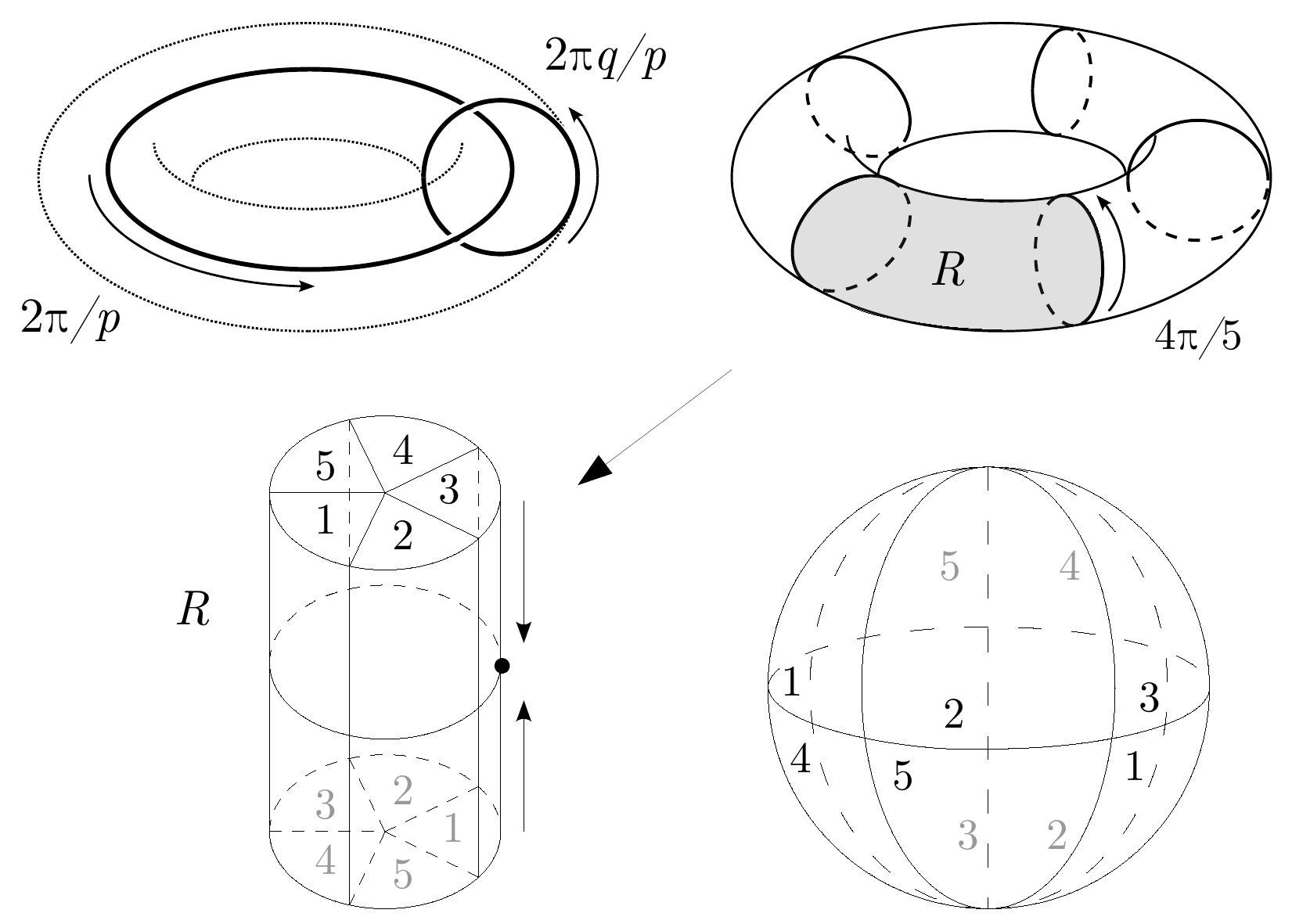}
\caption[legenda elenco figure]{Lens space $L(5,2)$ from the solid torus model of $\s3$.}\label{torusmodel}
\end{center}
\end{figure}

The proof of the equivalence of these two constructions can be found in \cite{WA}, and since it is relevant for our purpose, we can recall it briefly here.
The construction of $\s3$ as the join of two circles is the following: $\s3 = \s1 \times \s1 \times [0,1] / \sim_{J}$, where $\sim_{J}$ is the equivalence relation defined by $(a_{1},b,0) \sim_{J} (a_{2},b,0) $ for all $ a_{1}, a_{2} \in \s1, b \in \s1$ and $(a,b_{1},1) \sim_{J} (a,b_{2},1) $ for all $a \in \s1, b_{1}, b_{2} \in \s1$. It is essential to visualize the two circles in a Hopf configuration. Let $D^{2}= \{ z \in \mathbb{C} \ | \ ||z||≤1 \}$ be the unitary disk. This model of $\s3$ is equivalent to the following: considering the solid torus $\s1 \times D^{2}$, for each $Q \in \s1=\partial D^2$, each parallel $\s1 \times \{ Q \}$ of the torus $\s1 \times \partial D^{2}$ collapses to the point $Q$. Under this equivalence, the first circle of the join can be thought of as $\s1 \times \{ 0 \}$ while the second circle can be thought of as $\{ P \} \times \partial D^{2}/ \sim _{J}$, with $ P  \in \s1$.

The effect of the action of $G_{p,q}$ on this model of $\s3$ is the following: the circle $\mathfrak{l}= \s1 \times \{ 0 \}$ of the solid torus is rotated by $2 \pi / p$ radians, thus we identify $p$ equidistant copies of a meridian disk. The second $\s1$, visualized as a meridian $\mathfrak{m} = \{P \} \times \partial D^{2}$ of the torus, is rotated by $2 \pi q/ p$ radians, thus each of the $p$ copies of the meridian disk is identified with a rotation of $2 \pi q/ p$ radians.

As Figure \ref{torusmodel} shows, a fundamental domain under this action is a cylinder $R=[0,1] \times D^{2}$ with identification on the boundary, precisely each segment $ [0,1] \times \{ Q \}$ (with $Q \in \partial D^{2}$) of the lateral surface collapses to the point $\{ 1/2 \}$, and the top and the bottom disks are identified with each other after a rotation of $2 \pi q/p$ radians; in this way we can recognize the first model of the lens space.

\subsection{The construction of the disk diagram}

In this paper all links in the lens space $L(p,q)$ are considered up to ambient isotopy and up to link's orientation. Since we are not interested in the case of $\s3$, we assume $p>1$. The definition of the disk diagram developed in \cite{CMM} is the following.

Let $L$ be a link in $L(p,q)$ and consider $L'=F^{-1}(L)$. By moving $L$ via a small isotopy in $L(p,q)$, we can suppose that:
\begin{enumerate} \itemsep-4pt
\item[i)] $L'$ does not meet the poles $N$ and $S$ of $B^{3}$;
\item[ii)] $L' \cap \partial B^{3}$ consists of a finite set of points;
\item[iii)] $L'$ is not tangent to $\partial B^3$;
\item[iv)] $L' \cap \partial B^{2}_{0} = \emptyset$.
\end{enumerate}

As a consequence, $L'$ is the disjoint union of closed curves in $\text{int} B^{3}$ and arcs properly embedded in $ B^{3}$.
Let $\mathbf{p} \colon B^{3} \smallsetminus \{ N,S \} \rightarrow B^{2}_{0}$ be the projection defined by $\mathbf{p}(x)=c(x) \cap B^{2}_{0}$, where $c(x)$ is the circle (possibly a line) through $N$, $x$ and $S$.
Take $L'$ and project it using $\mathbf{p}_{|L'} \colon L' \rightarrow B^{2}_{0}$. 
As in the classical link projection, taken a point $P \in \mathbf{p}(L')$, its counterimage $\mathbf{p}^{-1}(P)$ in $L'$ may contain more than one element; in this case we say that $P$ is either a \emph{double} or \emph{multiple} point.

We can assume, by moving $L$ via a small isotopy, that the projection $\mathbf{p}_{|L'} \colon L' \rightarrow B^{2}_{0}$ of $L$ is \emph{regular}, namely:
\begin{enumerate} \itemsep-4pt
\item[1)] the projection of $L'$ contains no cusps;
\item[2)] all auto-intersections of $\mathbf{p}(L')$ are transversal;
\item[3)] the set of multiple points is finite, and all of them are actually double points;
\item[4)] no double point is on $\partial B^{2}_{0}$.
\end{enumerate}

Finally, double points are resolved by underpasses and overpasses as in the diagram for links in $\s3$.
A \emph{disk diagram} of a link $L$ in $L(p,q)$ is a regular projection of $L'=F^{-1}(L)$ on the equatorial disk $B^{2}_{0}$, with specified overpasses and underpasses.

\begin{figure}[h!]                      
\begin{center}                         
\includegraphics[width=11.2cm]{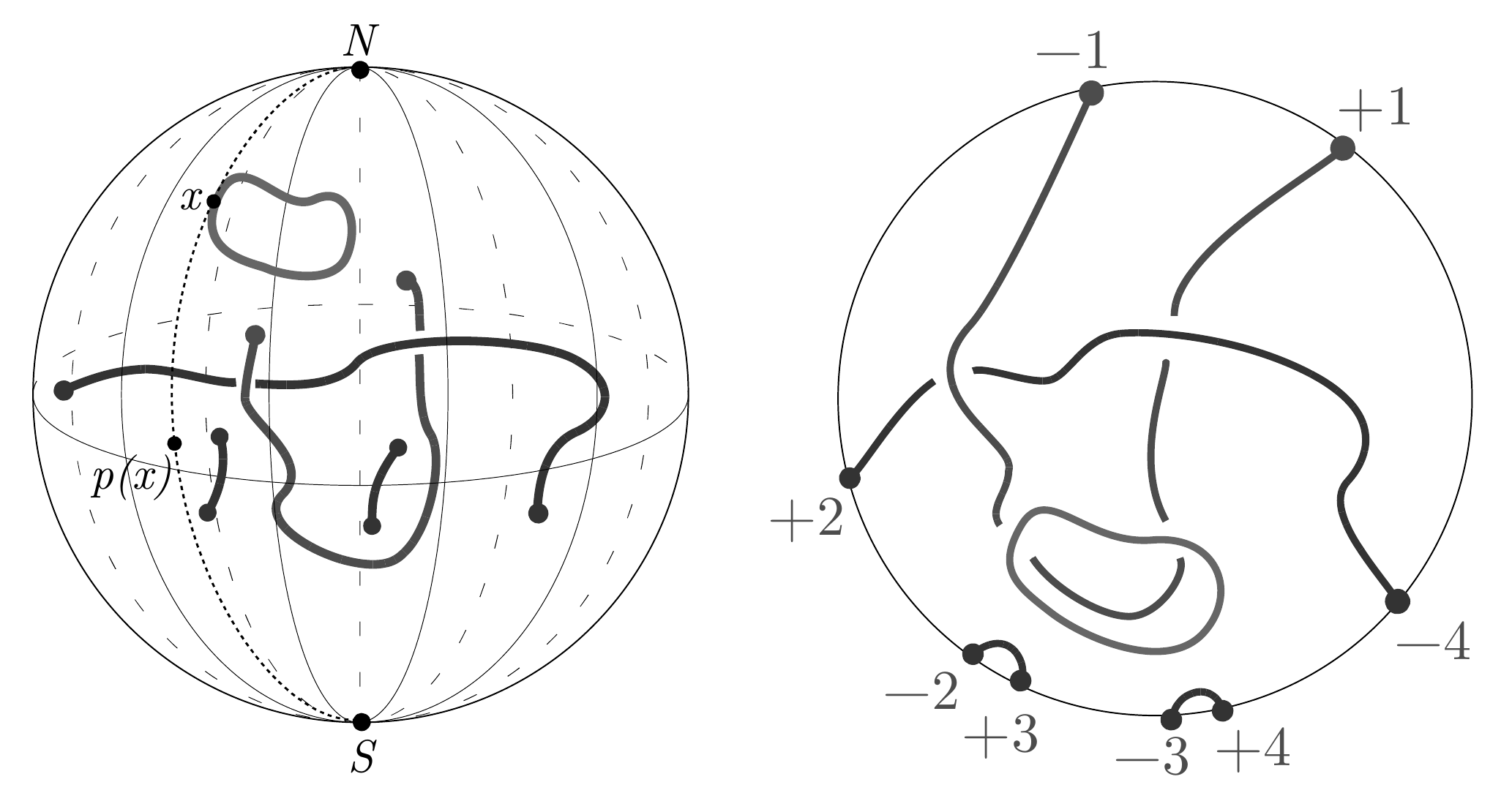}
\caption[legenda elenco figure]{A link in $L(9,1)$ and its corresponding disk diagram.}\label{link3}
\end{center}
\end{figure}

In order to have a more comprehensible diagram, we index the boundary points of the projection as follows: first, we assume that the equator $\partial B^{2}_{0}$ is oriented counterclockwise if we look at it from $N$, then, according to this orientation, we label with $+1, \ldots, +t$ the endpoints of the projection of the link coming from the upper hemisphere, and with $-1, \ldots, -t$ the endpoints coming from the lower hemisphere, respecting the rule $+i \sim -i$. An example is shown in Figure~\ref{link3}.

In \cite{CMM} it is shown that two disk diagrams of links in lens space represent equivalent links if and only if they are connected by a finite sequence of the seven Reidemeister type moves illustrated in Figure \ref{R17}.

\begin{figure}[h!]                      
\begin{center}                         
\includegraphics[width=14cm]{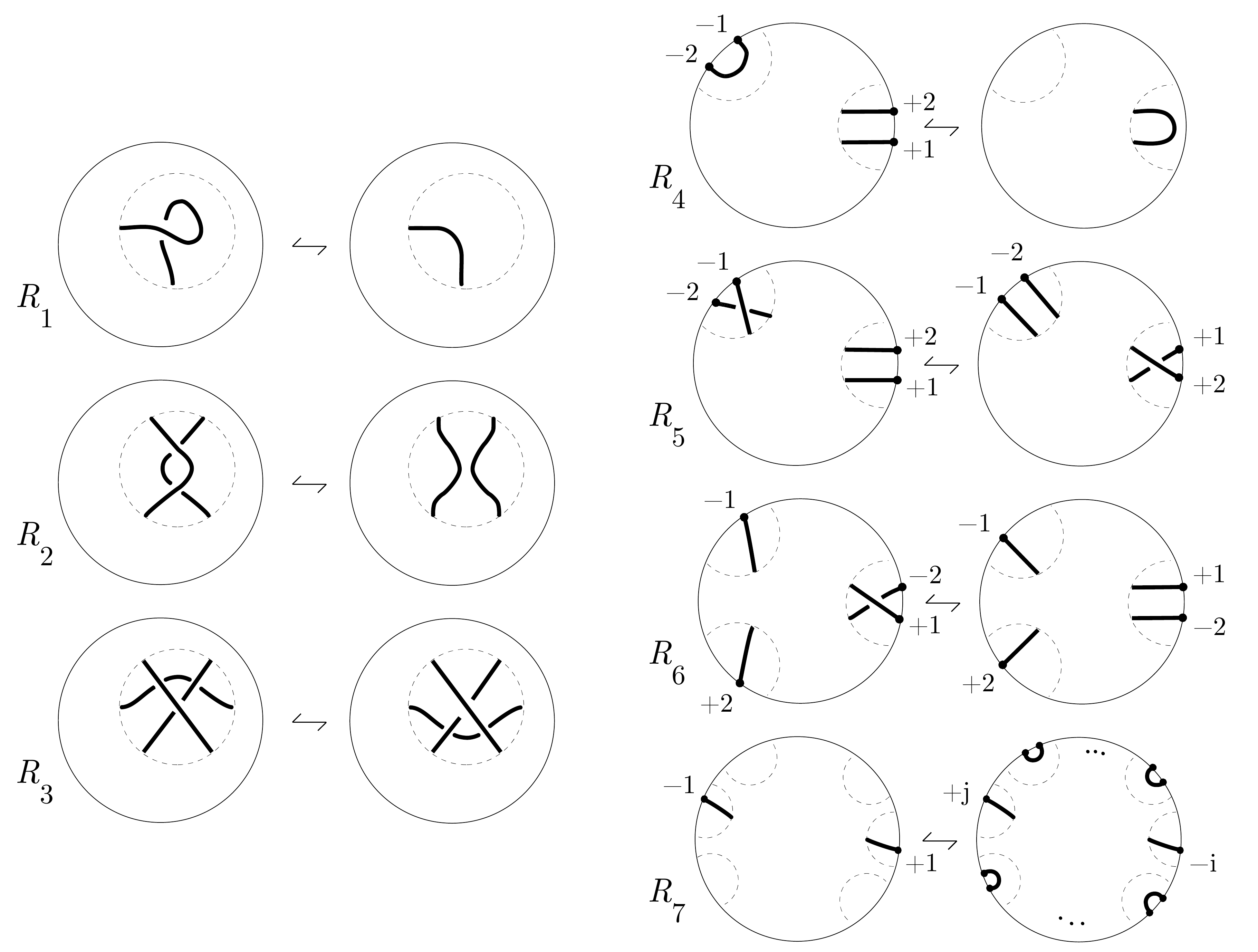}
\caption[legenda elenco figure]{Generalized Reidemeister moves.}\label{R17}
\end{center}
\end{figure}

A disk diagram is defined \emph{standard} if the labels on its boundary points, read according to the orientation on $\partial B^{2}_{0}$, are $(+1, \ldots, +t, -1, \ldots, -t)$. 

\begin{prop}\label{standard}
Every disk diagram can be reduced to a standard disk diagram with some small isotopies:
if $p=2$, the signs of its boundary points can be exchanged;
if $p>2$, a finite sequence of $R_{6}$ moves can be applied in order to bring all the plus-type boundary points aside.
\end{prop}
 
\begin{proof}
For $p=2$, the exchange of the signs of a boundary point corresponds to a small isotopy on the link, that crosses the equator of $B^3$. For $p>2$, the following strategy has to be considered.
By definition, the endpoints $+1, \ldots, +t$ on the boundary are always in this order if we forget the minus-type points. 
The endpoints $+i$ and $-i$ can be moved together along the boundary, with their respective arcs. Moreover we can assume that this small isotopy is performed close enough to the boundary as to avoid crossings.
Our aim is to bring all the plus-type boundary points one aside the other, respecting their labeling order. The isotopy performed can exchange $+i$ and $-j$ producing an $R_{6}$ move. 
Sometimes also the $-i$ endpoint may be exchanged with a $+k$ endpoint, producing an opposite $R_6$ move, that is to say, a move that creates one crossing. Consider the following algorithm: fix $+1$ and $-1$, bring $+2$ next to $+1$ (and hence $-2$ next to $-1$), bring $+3$ next to $+2$ ($-3$ next ot $-2$) and so on. If we apply it,  then an opposite $R_6$ move is always canceled by a subsequent $R_6$ move, that is to say, to get a standard disk diagram is enough to perform a sequence of $R_6$ moves. See Figure \ref{exR6} for an example. 
\end{proof}
 
\begin{figure}[h!]                      
\begin{center}                         
\includegraphics[width=14cm]{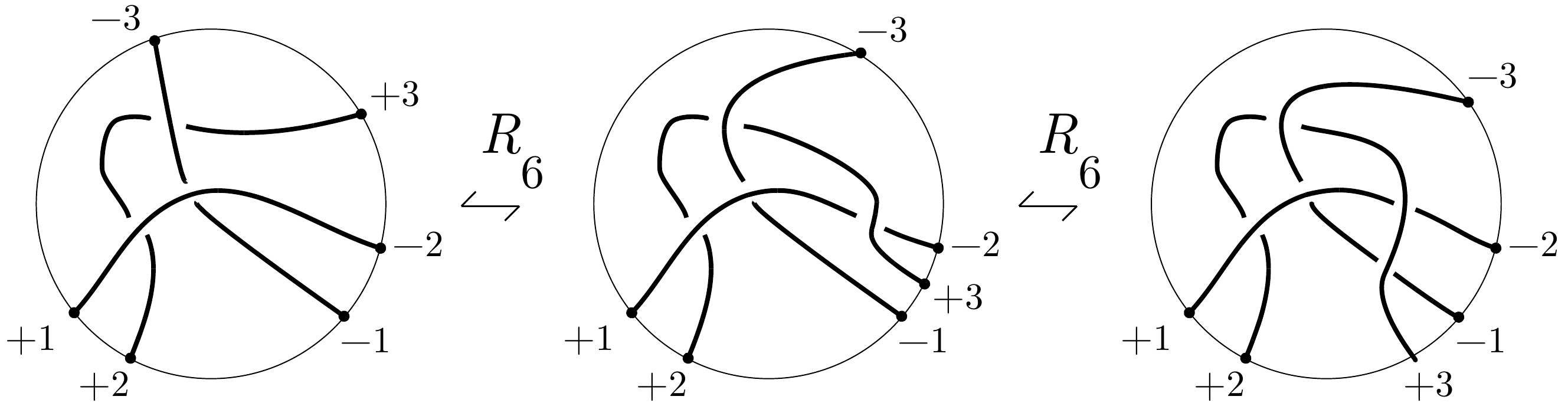}
\caption[legenda elenco figure]{Example of $R_{6}$-reduction to standard disk diagram.}\label{exR6}
\end{center}
\end{figure}

\subsection{Lift of links}
Let $L$ be a link in the lens space $ L(p,q)$; we denote by $\widetilde L=P^{-1}(L)$ the \emph{lift} of $L$ in $\s3$ under the quotient map $P \colon \s3 \rightarrow L(p,q)$.

Let $L$ be a link in $L(p,q)$, denote with $\nu$ its number of components, and with $\delta_{1}, \ldots, \delta_{\nu}$ the homology class in $H_{1}(L(p,q)) \cong \mathbb{Z}_{p}$ of the $i$-th component $L_{i}$ of $L$. In \cite{CMM} it is described a method that allows the computation of the homology classes from the disk diagram.

\begin{prop}\label{comp}
Given a link $L \subset L(p,q)$, the number of components of $\widetilde L$ is \vspace{-5mm}
$$\sum_{i=1}^{\nu} \gcd (\delta_{i},p).$$
\vspace{-10mm}
\end{prop}
\begin{proof}
The covering $\s3 \rightarrow L(p,q)$ is cyclic of order $p$, so that each component $L_{i}$ of $L$ has lift $\widetilde L_{i}$ with $\gcd ( \delta_{i}, p )$ components. As a consequence, if we sum over all the components of $L$, the lift $\widetilde L$ has $\sum_{i=1}^{\nu} \gcd (\delta_{i},p)$ components.
\end{proof}

The construction of a diagram for $\widetilde L \subset \s3$ starting from a disk diagram of $L \subset L(p,q)$ is explained by the following two theorems. The case of \mbox{$L(2,1) \cong \rp3$} is outlined in \cite{D}. Before stating the theorems, we must not forget the following notation about braids.
Let $B_{t}$ be the braid group on $t$ letters and let $\sigma_{1}, \ldots, \sigma_{t-1}$ be the Artin generators of $B_{t}$. Consider the Garside braid $\Delta_{t}$ on $t$ strands defined by $ (\sigma_{t-1}\sigma_{t-2} \cdots \sigma_{1})( \sigma_{t-1}\sigma_{t-2} \cdots \sigma_{2}) \cdots (\sigma_{t-1}) $ and illustrated in Figure \ref{treccia}. This braid can be seen also as a positive half-twist of all the strands and it belongs to the center of the braid group, that is to say, it commutes with every braid. Moreover $\Delta_{t}^{-1}$ can be represented in the braid group by the word $ (\sigma_{t-1}^{-1}\sigma_{t-2}^{-1} \cdots \sigma_{1}^{-1})( \sigma_{t-1}^{-1}\sigma_{t-2}^{-1} \cdots \sigma_{2}^{-1}) \cdots (\sigma_{t-1}^{-1}) $.

\begin{figure}[h!]                      
\begin{center}                         
\includegraphics[width=12cm]{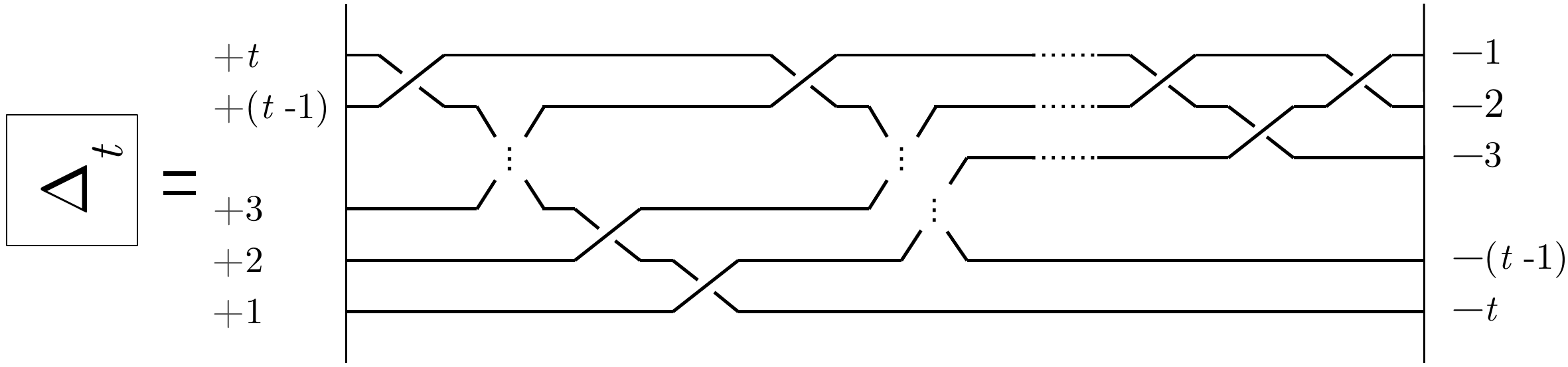}
\caption[legenda elenco figure]{The braid $\Delta_{t}$.}\label{treccia}
\end{center}
\end{figure}

\begin{teo}\label{teolift1}
Let $L$ be a link in the lens space $L(p,q)$ and let $D$ be a standard disk diagram for $L$; then a diagram for the lift $\widetilde L \subset \s3$ can be found as follows (refer to Figure \ref{diag}):
\begin{itemize}
\item consider $p$ copies $D_{1}, \ldots, D_{p}$ of the standard disk diagram $D$;
\item for each $i=1, \ldots, p-1$, using the braid $\Delta_{t}^{-1}$, connect the diagram $D_{i+1}$ with the diagram $D_{i}$, joining the boundary point $-j$ of $D_{i+1}$ with the boundary point $+j$ of $D_{i}$;
\item connect $D_{1}$ with $D_{p}$ via the braid $\Delta_{t}^{2q-1}$, where the boundary points are connected as in the previous case.
\end{itemize}

\begin{figure}[h!]                      
\begin{center}                         
\includegraphics[width=12cm]{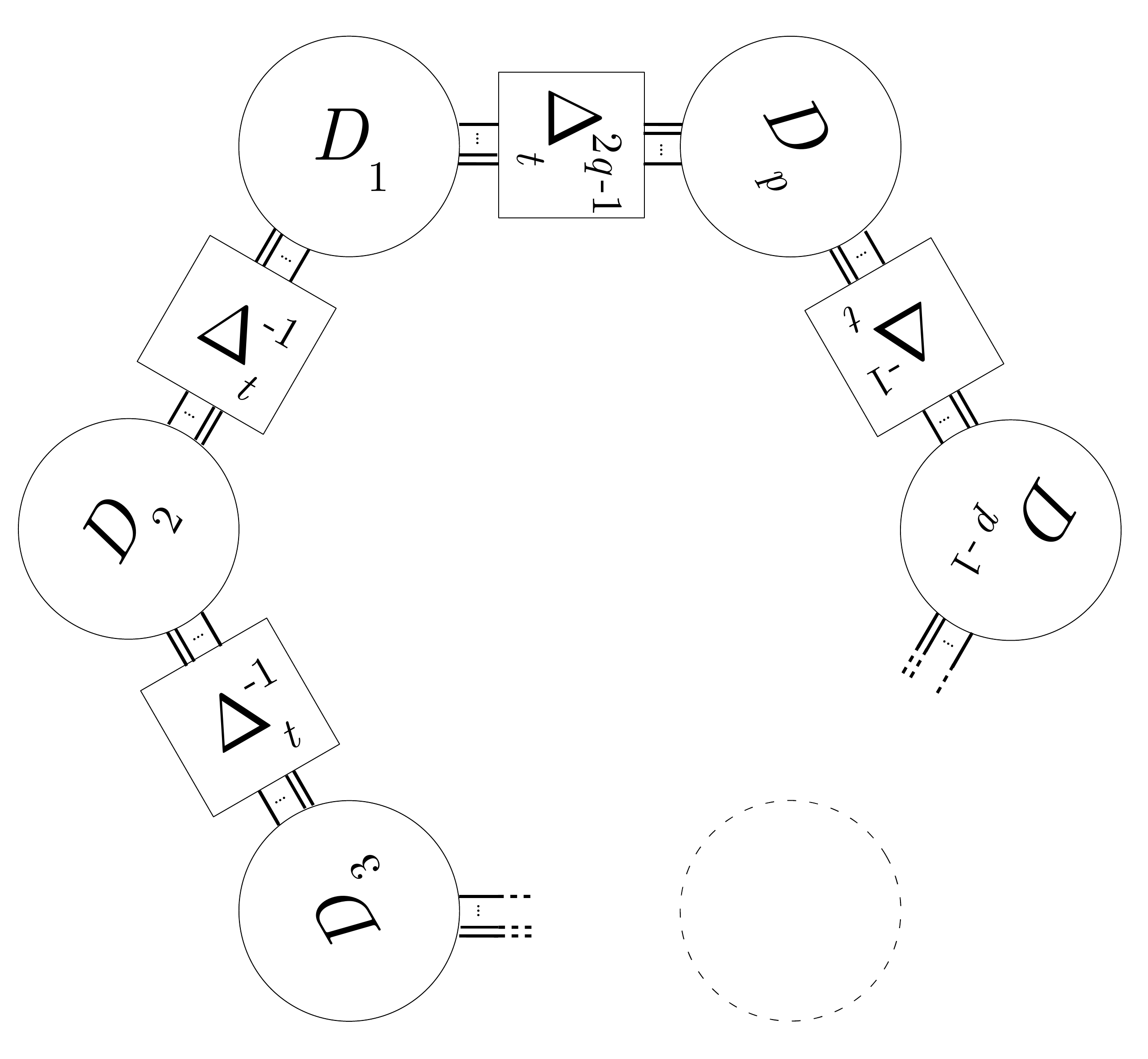}
\caption[legenda elenco figure]{Diagram of the lift in $\s3$ of a link in $L(p,q)$.}\label{diag}
\end{center}
\end{figure}
\end{teo}

\begin{proof}

Let $L$ be a link in $L(p,q)$ and let $D$ be a standard disk diagram for it. 
The lift in $\s3$ can be obtained from the model of $\s3$ where the solid torus has each parallel which collapses into a point. From this model the lens space $L(p,q)$ is described as in Section \ref{sub1}, so we can embed into the solid torus the $p$ copies $D_1, \ldots, D_p$ of the standard disk diagram $D$ in $L(p,q)$. The $p$ copies of the diagram are embedded as disks bounded by a meridian. Each of them is rotated by $2 \pi q/p$ radians around $\mathfrak{l} =\s1 \times \{ 0 \}$, with respect to the previous copy of the diagram. By this rotation, if you consider the parallel $\s1 \times \{ Q \}$ on the boundary of the torus that passes through the endpoint $+j$ of $D_{i}$, then it passes also through $-j$ of $D_{i+1}$. In the solid torus model, each of these parallels collapses to a point, so that all the pairs previously described are identified. If we want to show this identification, we can draw on our torus each arc of the parallel from $+j \in D_{i}$ to $-j \in D_{i+1}$, as Figure \ref{liftS3} shows, obtaining a representation for the lift $\widetilde L$ in the solid torus model of $\s3$.
\begin{figure}[h!]                      
\begin{center}                         
\includegraphics[width=12cm]{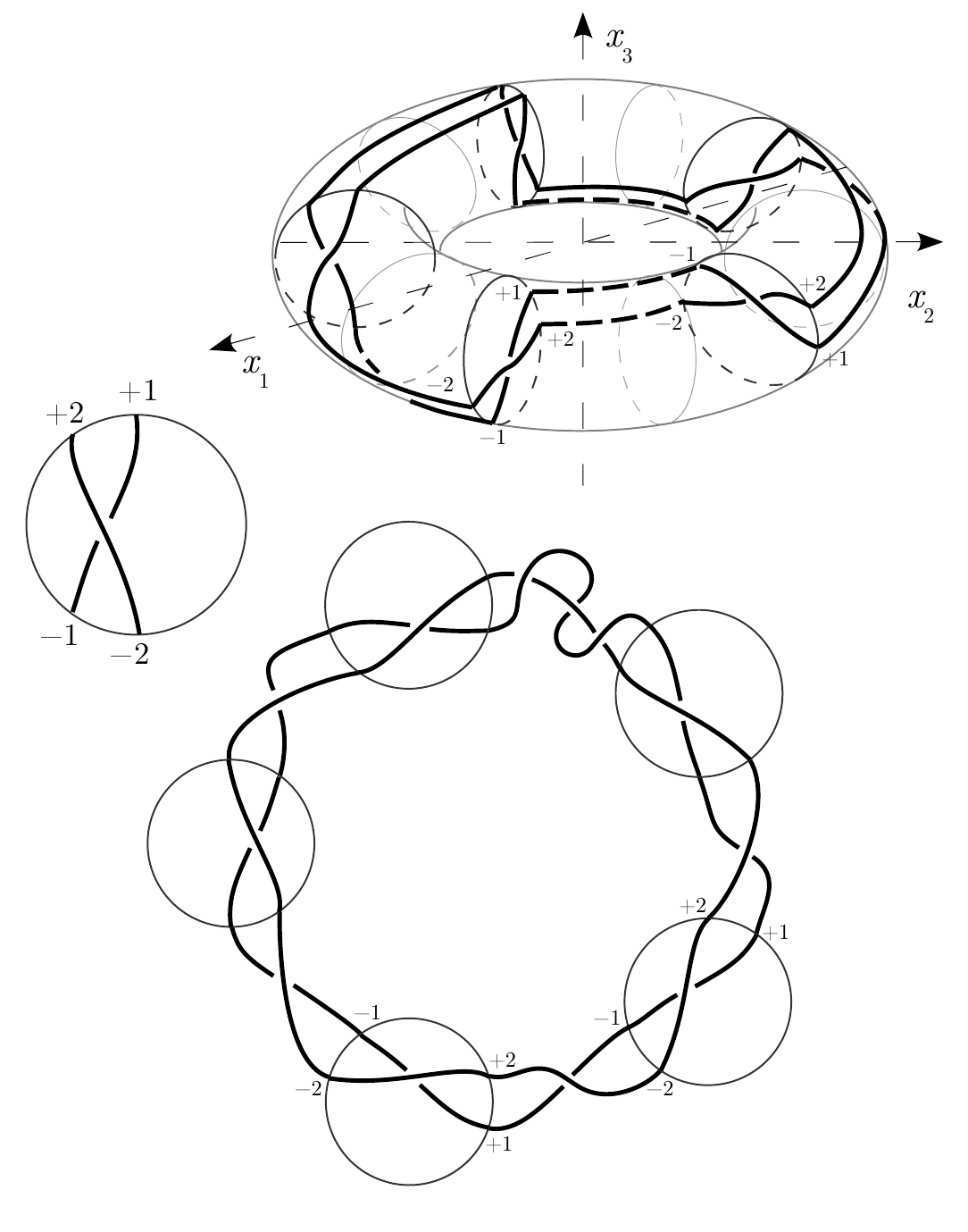}
\caption[legenda elenco figure]{Lift in $\s3$ of a link in $L(5,2)$.}\label{liftS3}
\end{center}
\end{figure}

In order to get a planar diagram for $\widetilde L$ that comes from this representation, we can do as follows. Put $\s1 \times D^{2}$ into $\mathbb{R}^{3}$ and fix cartesian axis $(x_1,x_2,x_3)$, where $x_3$ is orthogonal to the plane containing $\s1$. For each copy $D_{i}$ of $D$, consider its intersection with the plane $\{ x_3=0\}$ and rotate $D_{i}$ around this diameter by $\pi/2$ radians, so that $D_{i}$ is turned upward. As a result, the connection lines between the two disks $D_{i}$ and $D_{i+1}$ are braided by $\Delta_{t}^{-1}$ in order to avoid the projection of the two disks. 
Furthermore, when a toric braid, twisting around the core of $2 \pi q$, becomes planar, we have to add another piece of braid, namely $\Delta_{t}^{2q}$. 
In this way we will have exactly the planar diagram of Figure \ref{diag}. 
\end{proof}

\begin{oss}
The lift in $\s3$ of a link $L\subset L(p,q)$ is exactly a \emph{$(p,q)$-lens link} in $\s3$, according to \cite{C4}. Precisely, the $n$-tangle $T$ that Chbili uses in his construction is the composition of the disk diagram $D$ of $L \subset L(p,q)$ with the braid $\Delta_{t}^{-1}$.
\end{oss}

The previous planar diagram of the lift has not got the least possible number of crossings. Indeed if, in the last step of the previous proof, we rotate $D_{1}$ of $\pi/2$ radians and $D_{2}$ of $-\pi / 2 $ radians around the diameter of the diagram, we avoid the braid $\Delta_{t}^{-1}$ between the two disks. We now explain how to get a diagram with fewer crossings. First of all, let us define the reverse disk diagram $\overline D$ of $D$: consider the symmetry of $D$ with respect to an external line and then exchange all overpasses/underpasses. 

\begin{prop}\label{another}
Let $L$ be a link in the lens space $L(p,q)$ and let $D$ be a standard disk diagram for $L$; then a diagram for the lift $\widetilde L \subset \s3$ can be found as follows (refer to Figure \ref{DiagSollD}):
\begin{itemize}
\item consider $p$ copies $D_{1}, \ldots D_{p}$ of the standard disk diagram $D$, then denote $F_{i}=D_{i}$ if $i$ is odd, and $F_{i}=\overline D_{i}$ if $i$ is even;
\item for each $i=1, \ldots, p-1$, using a trivial braid, connect the diagram $F_{i+1}$ with the diagram $F_{i}$ joining the boundary point $-j$ of $D_{i+1}$ with the boundary point $+j$ of $D_{i}$;
\item connect $D_{1}$ with $D_{p}$ via the braid $\Delta_{t}^{2q-p}$, where the boundary points are connected as in the previous case.
\end{itemize}
\begin{figure}[h!]                      
\begin{center}                         
\includegraphics[width=10cm]{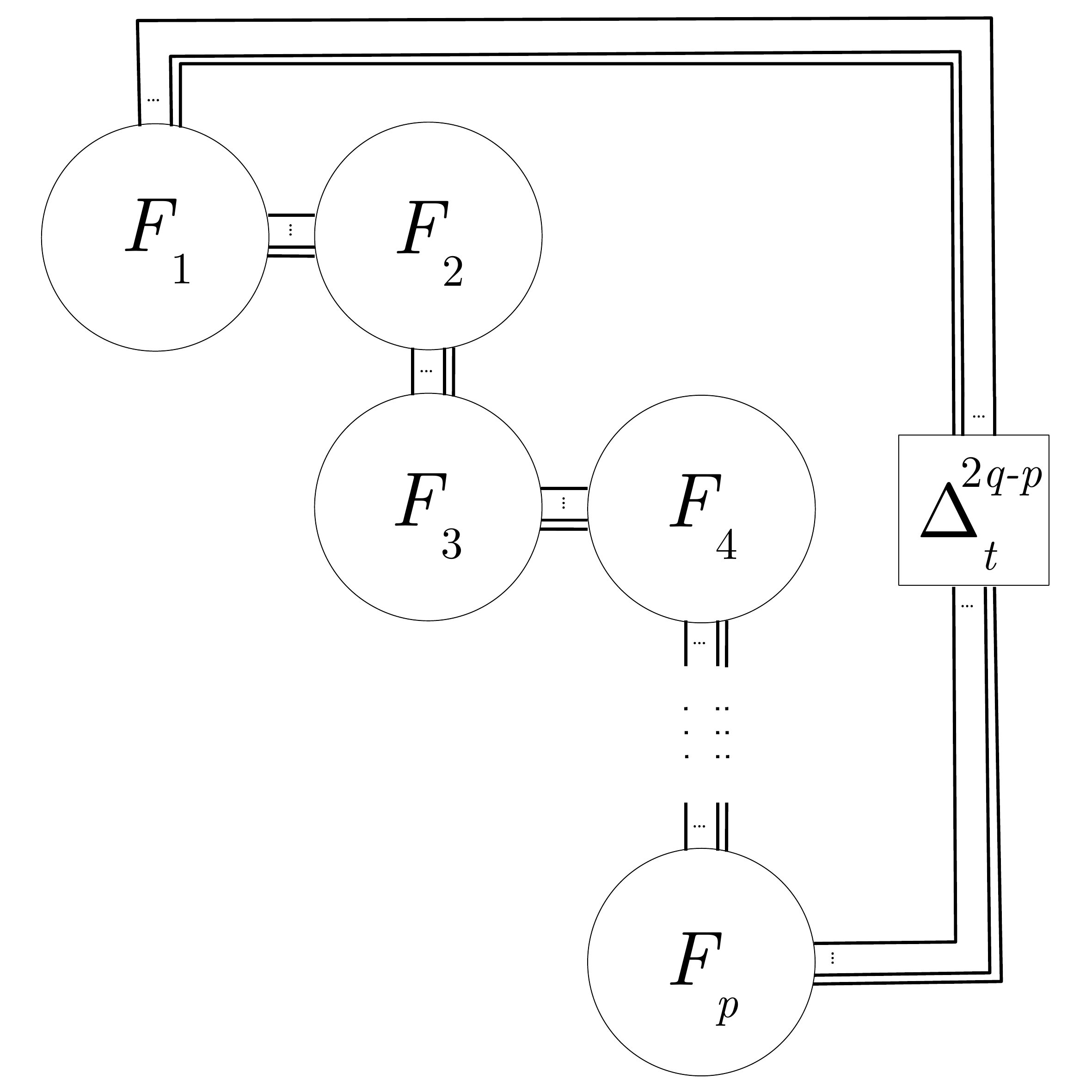}
\caption[legenda elenco figure]{Another diagram of the lift in $\s3$ of a link in $L(p,q)$.}\label{DiagSollD}
\end{center}
\end{figure}
\end{prop}

Please refer to Figure \ref{linktreccia} for an example of diagram of the lift.

\begin{proof}
Consider the planar diagram of the lift of Theorem \ref{teolift1} and comb it, reversing upside down $D_{2}$, reversing two times $D_{3}$, three times $D_{4}$ and so on. The odd-index diagrams are unchanged and all the even-index diagrams become $\overline D_{2}, \overline D_{4}, \ldots$ in the new diagram of the lift. The $p-1$ braids $\Delta_{t}^{-1}$ between the disks are shifted near the braid $\Delta_{t}^{2q-1}$, so that you get $\Delta_{t}^{2q-p}$ in this new form of the diagram, reducing the number of crossings.
\end{proof}

\end{section}


\begin{section}{Lift of families of knots and links}

In this section we show the behavior of the lift for several knot constructions like split links, composite knots and braid links. Remember that a knot is \emph{trivial} if it bounds a $2$-disk in $L(p,q)$ and that a link $L \subset L(p,q)$ is \emph{local} if it is contained inside a $3$-ball. The disk diagram of a local link, up to generalized Reidemeister moves, can avoid $\partial B^{2}_{0}$. As a consequence of Theorem \ref{teolift1}, a local link is lifted to $p$ disjoint copies of itself.

\subsection{Split, composite and satellite links} 
The definition of split links in $\s3$ can be generalized to lens spaces: a link $L \subset L(p,q)$ is \emph{split} if there exists a $2$-sphere in the complement $L(p,q) \smallsetminus L$ that separates one or more components of $L$ from the others. 
The $2$-sphere separates $L(p,q)$ into a ball $\hat B^3$ and $\overline{L(p,q) \smallsetminus \hat B^3 }$; as a consequence, a split link is the disjoint union of a local link and of another link in lens space. 
If we consider the lift of a split link $L= L_1 \sqcup L_2$, where $L_1 \subset \hat B^3$ and $L_2 \subset L(p,q) \smallsetminus \hat B^3$, then $L_1$ are lifted to $p$ split copies of $L_1$ and $L_2$ are lifted to some link $\widetilde L_2$. In formulae: $$\widetilde L=  \underbrace{L_1 \sqcup \ldots \sqcup  L_1}_{p}  \sqcup \ \widetilde L_2 .$$

We can easily generalize the definition of satellite knot to lens space, following Section C, Chapter 2 of \cite{BZ}.
Take $K_p$ a knot in the solid torus $T$ that is neither contained inside a $3$-ball nor the core of the solid torus, and call it \emph{pattern}. Let $e \colon T \rightarrow L(p,q)$ be an embedding such that e(T) is the tubular neighborhood of a non-trivial knot $K_c \subset L(p,q)$. The knot $K:=e(K_p) \subset L(p,q)$ is the \emph{satellite} of the knot $K_c$, called \emph{companion} of $K$. The satellite of a link can be constructed by specifying the pattern of each component. In addition the pattern of a satellite knot can be a link too.
A \emph{cable} knot is a satellite knot with a torus knot as pattern.
We do not have explicit formulae for the lift of satellite or cable knots, but Example \ref{CE3} helps us to understand the behavior of the lift.


\emph{Composite} knots are a special case of satellite knots, that is to say, satellite knots where the pattern, up to isotopy in $T$, has the following two properties: there exists a meridian of $T$ such that the disk bounding the meridian intersects the pattern in a single point, moreover the pattern must not be isotopic to the core of the solid torus. The notation in this case becomes $K=K_c \sharp K_p$ (\emph{connected sum}), and $K_p$ can be seen also as a knot in $\s3$.

Let $K_1 \subset L(p,q)$ be a \emph{primitive}-homologous knot, that is to say, a knot whose homology class in $H_{1}(L(p,q))$ is coprime with $p$ (we require this because, according to Proposition \ref{comp}, its lift is a knot). Let $K_2 \subset \s3$ be a knot. Then the lift $\widetilde K$ of the connected sum $K=K_1 \sharp K_2$ is $$\widetilde K=\widetilde K_1 \sharp \underbrace{K_2 \sharp \ldots \sharp K_2}_{p}.$$
This formula can be proved in the following way: up to generalized Reidemeister move, we can suppose that the disk diagram of $K_1 \sharp K_2$ has the projection of $K_2$ all contained in a disk inside $B^{2}_{0}$, therefore from the diagram of Theorem \ref{teolift1} we can easily see the result.

In order to define the connected sum for links we have to specify the component of each link to which we add the pattern.
If we consider a knot $K_1 \subset L(p,q)$ such that $\gcd([K_1],p) \neq 1$ or a link $L_1$ with more than one component, then, because of Proposition \ref{comp}, its lift has more than one component. In this case the lift can be found selecting the components of $\widetilde K_1$ or $\widetilde L_1$ where the copies of $K_2$ have to be connected.

As a consequence, if a link $L \subset L(p,q)$ is composite, then also its lift $\widetilde L \subset \s3$ is composite. That is to say, if $\widetilde L$ is prime, then we know that $L$ is prime too. 

\subsection{Links in lens spaces from braids}
We can construct a link $L \subset L(p,q)$ starting from a braid $B$ on $t$ strands by considering the standard disk diagram where the braid $B$ has its two ends of the strands on the boundary, indexed respectively by the points $( +1, \ldots, +t)$ and $( -1, \ldots, -t )$. See Figure \ref{linktreccia2} for an example. In this case, we say that $B$ represents $L$.

\begin{figure}[h!]                      
\begin{center}                         
\includegraphics[width=9cm]{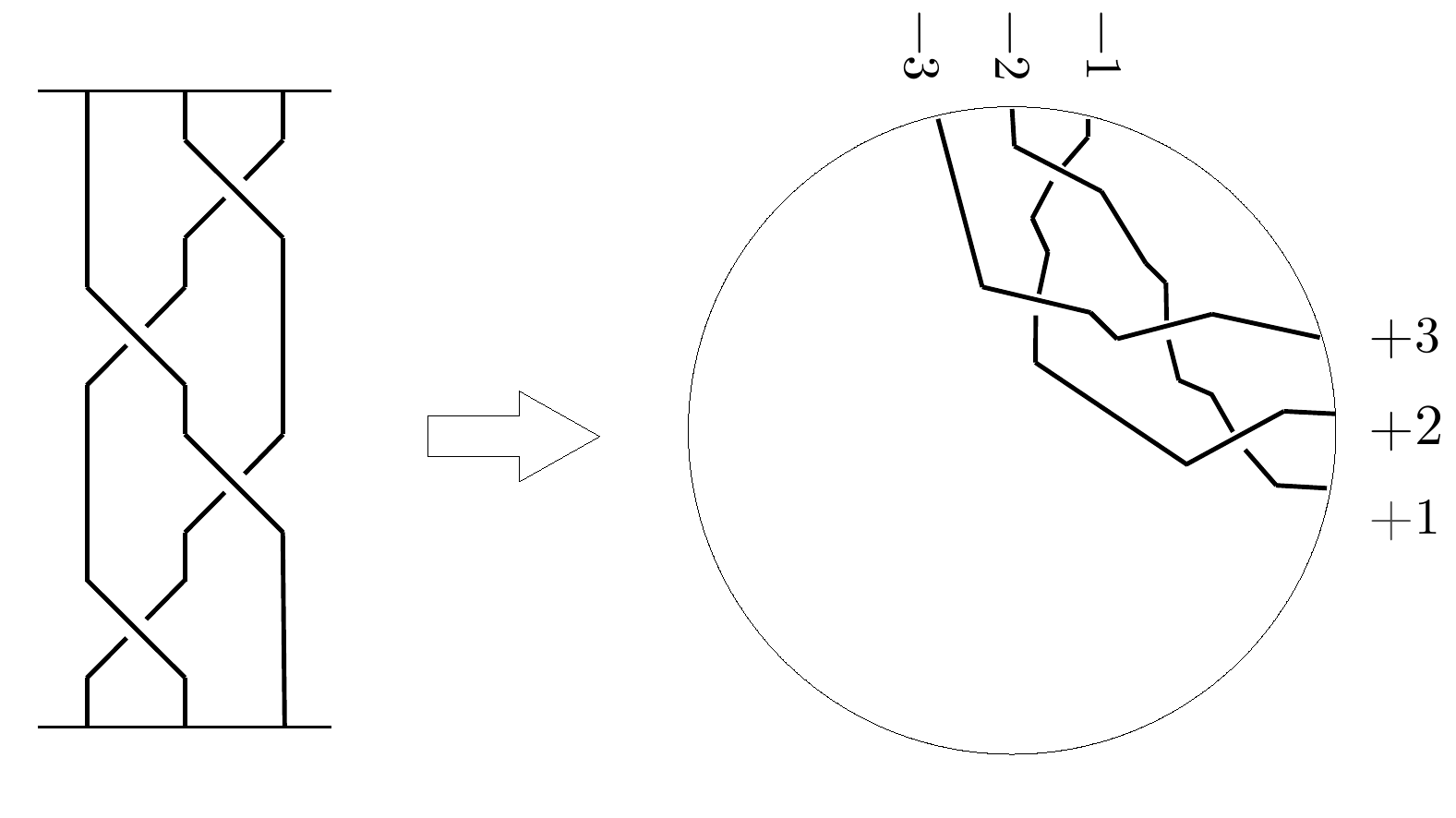}
\caption[legenda elenco figure]{The braid $B=\sigma_2\sigma_1\sigma_2\sigma_1$ becomes a standard disk diagram.}\label{linktreccia2}
\end{center}
\end{figure}

\begin{prop}
If $L \subset L(p,q)$ is a link represented by the braid $B$ on $t$ strands, then $\widetilde L$ is the link obtained by the closure in $\s3$ of the braid $(B \Delta_{t}^{-1})^{p}_{\vphantom{t}}  \Delta_{t}^{2q}$ or equivalently $B^{p}_{\vphantom{t}}  \Delta ^{2q-p}_{t}$.
\end{prop}

\begin{proof}
Using Theorem \ref{teolift1}, we replace the $p$ copies of the disk diagram $D$ with the braid $B$ representing the link. The result is the closure of the braid $(B \Delta_{t}^{-1})^{p}_{\vphantom{t}}  \Delta_{t}^{2q}$ in $\s3$, that can be transformed also into the braid $B^{p}_{\vphantom{t}}   \Delta ^{2q-p}_{t}$, since $\Delta_{t}$ is an element that belongs to the center of the braid group.
\end{proof}

\begin{oss}
The braid $B^{p}_{\vphantom{t}}  \Delta_{t}^{2q-p}$ is exactly the \emph{$(p,q)$-lens braid} of \cite{C3}.
\end{oss}

Which links in lens spaces are lifted to torus links? We have the following result, stated in \cite{C4}, that generalizes a result of \cite{H} for torus knots. Remember that the torus link $T_{n,m} \subset \s3$ is the closure of the braid $(\sigma_{1}\sigma_{2}\cdots \sigma_{n-1})^{m}$. 

\begin{prop}{\upshape\cite{C4}}
The torus link $T_{n,m}$ is a $(p,q)$-lens link (that is to say, it is the lift of some link in $L(p,q)$) if and only if $p$ divides $m-nq$.
\end{prop}

\begin{proof}
The torus link is the closure of the braid $(\sigma_{1}\sigma_{2}\cdots \sigma_{n-1})^{m}$ and the lift of our braid link is the closure of the braid $B^{p}  \Delta_{n}^{2q-p}$.
We know that in the braid group the element $\Delta_{n}^{2}$ can be represented by the word $ (\sigma_1 \cdots \sigma_{n-1})^n$.
Therefore the equality turns into $(\sigma_{1}\sigma_{2}\cdots \sigma_{n-1})^{m}=(B \Delta_{n}^{-1})^{p}  (\sigma_{1}\sigma_{2}\cdots \sigma_{n-1})^{nq}$ and the result is straightforward.
\end{proof}

\end{section}


\begin{section}{Different knots and links in lens spaces with equivalent lift.}

An invariant $I$ of links is \emph{complete} if $I(L_{1})=I(L_{2})$ implies that $L_{1} $ and $ L_{2}$ are equivalent, where $L_{1}$ and $L_{2}$ are two links.
In knot theory, an invariant that is both complete and easy to compute is still unknown. We have several examples of complete invariants: the knot group for prime knots in $\s3$ (this is the corollary of the results contained in \cite{Wh}  and \cite{GL}), the fundamental quandle for knots in $\s3$ \cite{J, M}, the oriented fundamental augmented rack for links in $3$-manifolds \cite{FR} and so on.
On the contrary, all invariants easy to compute, such as Jones or Alexander polynomials, cannot distinguish some pairs of different links, that is to say, they are not complete.

In this section we use the braid construction of the lift to find different links in lens spaces with an equivalent lift, that is, to prove that the lift is not a complete invariant. 

\subsection{Counterexamples from braid tabulation}

Given a braid $B$, denote by $ \widehat{B}$ the link in $\s3$ obtained as the closure of $B$.
The first step is to understand whether the Garside braid produces equivalent links $\widehat{\Delta_{t}^{k}} \subset \s3$ for different $t$ and $k$. The computations are summed up in Table \ref{LIFT}; the labels of the links are the one of the Knot Atlas \cite{KA}.

\begin{table}[h!]
\begin{center}
\begin{tabular}{|c|c|c|}
\hline 
$t$ & $B$ & $\widehat{B}$  \\
\hline
$1$ & $\Delta_1^0$ & $0_1$  \\
\hline
$2$ & $\Delta_2^0$ & $0_1 \sqcup 0_1$  \\
\hline
$2$ & $\Delta_2^1$ & $0_1$  \\
\hline
$2$ & $\Delta_2^2$ & $L2a1$  \\
\hline
$2$ & $\Delta_2^3$ & $3_1$  \\
\hline
$2$ & $\Delta_2^4$ & $L4a1$  \\
\hline
$2$ & $\Delta_2^5$ & $5_1$  \\
\hline
$2$ & $\Delta_2^6$ & $L6a3$  \\
\hline
\end{tabular}
\hspace{2cm}
\begin{tabular}{|c|c|c|}
\hline 
$t$ & $B$ & $\widehat{B}$  \\
\hline
$3$ & $\Delta_3^0$ & $0_1 \sqcup 0_1 \sqcup 0_1$  \\
\hline
$3$ & $\Delta_3^1$ & $L2a1$  \\
\hline
$3$ & $\Delta_3^2$ & $L6n1$  \\
\hline
$3$ & $\Delta_3^3$ & $L9n15$  \\
\hline
$4$ & $\Delta_4^0$ & $0_1 \sqcup 0_1 \sqcup 0_1 \sqcup 0_1$  \\
\hline
$4$ & $\Delta_4^1$ & $L4a1$  \\
\hline
$5$ & $\Delta_5^0$ & $0_1 \sqcup 0_1 \sqcup 0_1 \sqcup 0_1 \sqcup 0_1$  \\
\hline
$5$ & $\Delta_5^1$ & $L8n3$  \\
\hline
\end{tabular}
\end{center}
\caption[legenda elenco figure]{Links arising from the closure of Garside braids.}
\label{LIFT}
\end{table}

Greater string numbers or greater powers give links outside standard tabulations.
Moreover, for negative powers, we obtain the link that is the mirror image of the link with the corresponding positive power. If the link is amphicheiral, like the trivial knot or the Hopf link (also denoted by $L2a1$), then the closure is the same. 

At this stage we are looking for a braid $\Delta^{k}_{t}$ representing a link in $L(p,q)$ such that its lift is one of the possibilities in Table~\ref{LIFT}. As a consequence of Proposition~\ref{another}, the lift is the link represented by the braid $\Delta_{t}^{k \cdot p} \Delta_{t}^{2q-p}$. Hence we look for solutions of the equation: $\Delta_{t}^{k \cdot p} \Delta_{t}^{2q-p}=\Delta_{t}^{h}$, where $h$ is the suitable power of $\Delta_{t}$ that gives us the desired lift. 

Now we list all the possible cases where the braid closures of Table~\ref{LIFT} are equivalent, the desired examples will rise from the following computations.

\begin{ese}{\textbf{Different knots in $L \left (p, \frac{p \pm 1}{2} \right )$ with trivial knot lift.}}\label{CE1}
The trivial knot can be obtained either as the closure of any power of $\Delta_{1}$ or as the closure of $\Delta_{2}^{\pm 1}$.
In the first case, the link in any lens space $L(p,q)$ represented by the braid on one single string is lifted to the trivial knot.
In the second case, namely $\Delta_{2}^{\pm1}$, we have to study the equation $\Delta_{2}^{k \cdot p} \Delta_{2}^{2q-p}=\Delta_{2}^{\pm 1}$, that is to say, $kp+2q-p=\pm 1$.
For the positive case $kp+2q-p=1$ we can obtain integer solutions with $0<q<p$ only for $k=0$, $p$ odd and $q=\frac{p+1}{2}$. For the negative case, we have $k=0$, $p$ odd and $q=\frac{p-1}{2}$.

If we look for a pair of different knots in the same $L(p,q)$, we have to restrict to $L\left(p, \frac{p\pm 1}{2} \right )$ with $p$ odd. 
Consider $K_1$ as the knot represented by the braid $\Delta_{1}=1_{1}$ and $K_2$ as the knot represented by the braid $\Delta_{2}^{0}=1_{2}$, they are illustrated in Figure \ref{disCE1}. Are $K_1$ and $K_2$ different knots?

\begin{figure}[h!]                      
\begin{center}                        
\includegraphics[width=10cm]{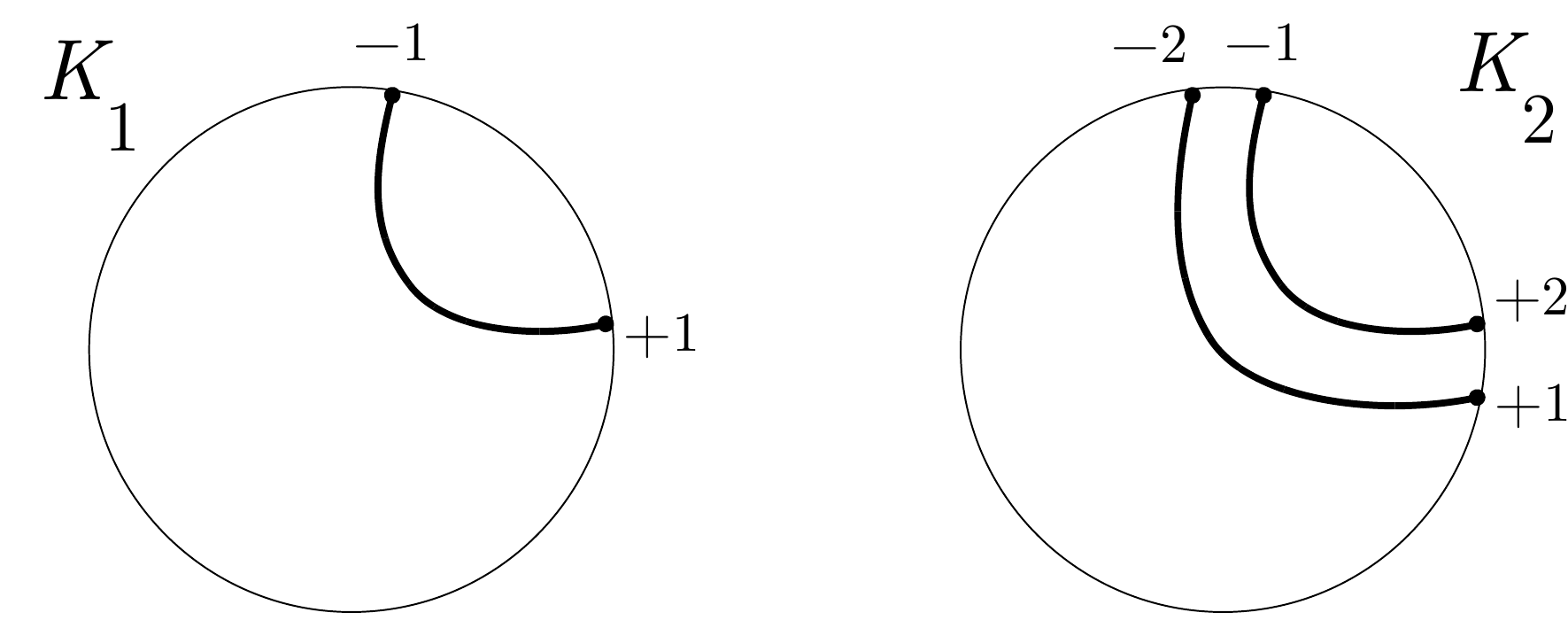}
\caption[legenda elenco figure]{Two different knots with equivalent lift in $L \left (p, \frac{p \pm 1}{2} \right )$.}\label{disCE1}
\end{center}
\end{figure}

The homology class $[K]= \delta$ of a knot in $L(p,q)$ can be \mbox{$0,1,\ldots p-1$,} but since we do not consider the orientation of the knots, we have to identify $\pm \delta$, so that the knots are partitioned into $\lfloor p/2 \rfloor +1$ classes: \mbox{$\delta=0,1,\ldots,  \lfloor p/2 \rfloor$,} where $\lfloor x \rfloor$ denotes the integer part of $x$. If two knots stay in different homology classes, they are necessarily different. The same reasoning holds also for links, with a more subtle partition.

Since $[K_1]=1$ and $[K_2]=2$, the two knots considered above in $L \left (p, \frac{p \pm 1}{2} \right )$ are different if $p>3$ and odd; if $p=3$ they are equivalent.

\end{ese}

\begin{ese}{\textbf{Different links in $L(4,1)$ with Hopf link lift.}}\label{CE}
As in the previous case, all the possible solutions of the corresponding equations are considered for the Hopf link $L2a1$. Table~\ref{L41} sums up the results.

\begin{table}[h!]
\begin{center}
\begin{tabular}{|c|c|c|}
\hline 
lift braid & equation & solutions  \\
\hline

  &  & for all $ p, \ L(p,1), \ k=1$  \\[-1.5ex] 
 \raisebox{2ex}{$\Delta_{2}^{2}$} &  \raisebox{2ex}{$kp+2q-p=2$} & for all $ p \textrm{ even}, \ L\left(p,\frac{p+2}{2} \right ), \ k=0$  \\
\hline

& & for all $  p, \ L(p,p-1), \ k=-1$  \\[-1.5ex]
 \raisebox{2ex}{$\Delta_{2}^{-2}$} &  \raisebox{2ex}{$kp+2q-p=-2$}  & for all $ p \textrm{ even}, \ L\left(p,\frac{p-2}{2} \right ), \ k=0$  \\
\hline

$\Delta_{3}^{1}$ & $kp+2q-p=1$ & for all $ p \textrm{ odd}, \ L\left(p,\frac{p+1}{2} \right ), \ k=0$  \\
\hline

$\Delta_{3}^{-1}$ & $kp+2q-p=-1$ & for all $ p \textrm{ odd}, \ L\left(p,\frac{p-1}{2} \right ), \ k=0$  \\

\hline
\end{tabular}
\end{center}
\caption[legenda elenco figure]{Links in lens spaces lifting to Hopf link.}
\label{L41}
\end{table}

We look for pairs of compatible solutions, and after excluding equivalent links, we get only the following pair of links in $L(4,1)$: consider the knot $L_A$ represented by the braid $B_{1}=1_{2}$ and the link $L_B$ represented by $B_{2}=\Delta_{2}$. They are different because they have a different number of components, but they have the same lift, the Hopf link. In order to better understand the topological construction of the lift, we illustrate it in Figure~\ref{linktreccia}.

\begin{figure}[h!]                      
\begin{center}                        
\includegraphics[width=14cm]{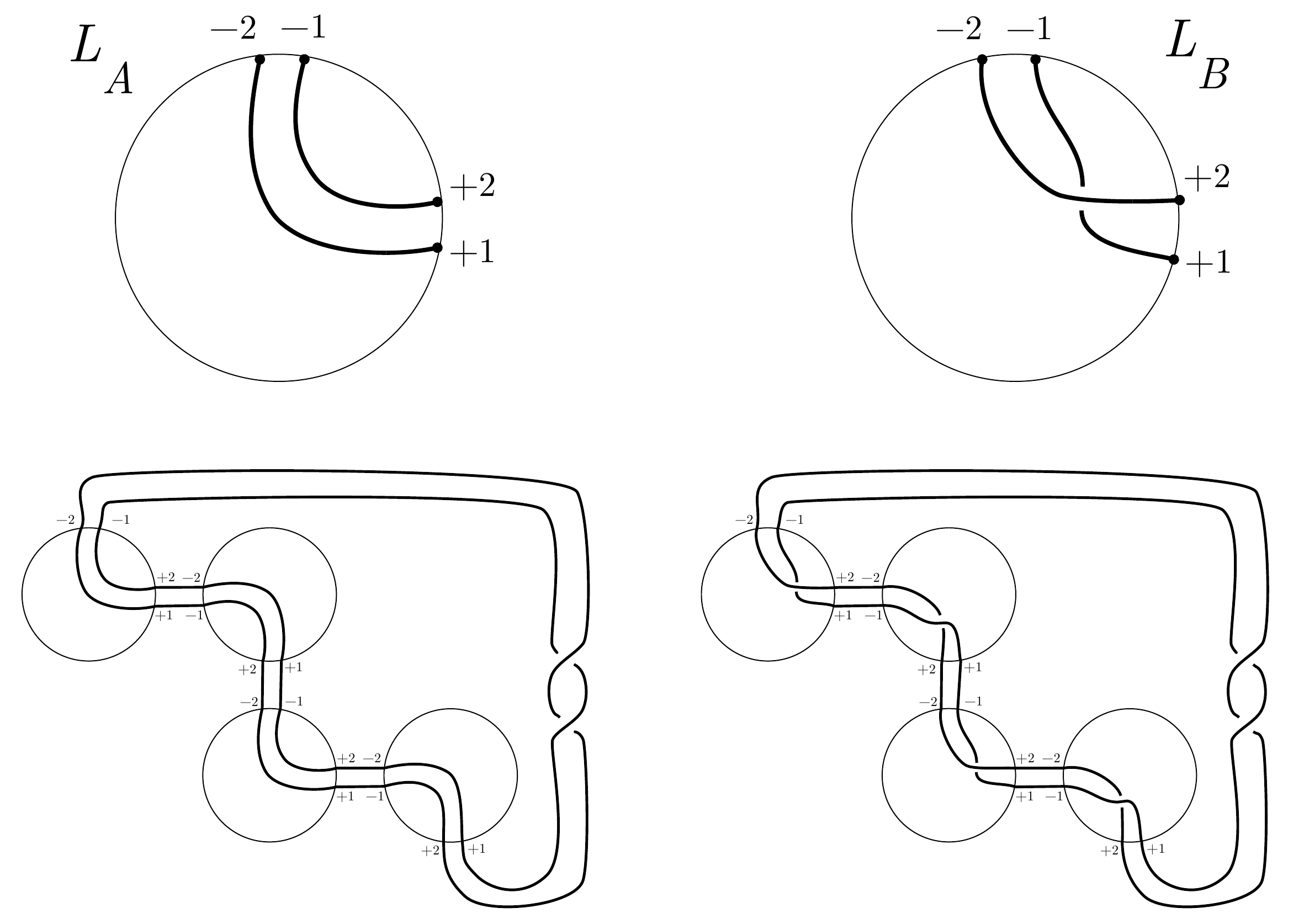}
\caption[legenda elenco figure]{Two different links with equivalent lift in $L(4,1)$.}\label{linktreccia}
\end{center}
\end{figure}
\end{ese}

The last case of Table~\ref{LIFT} is the link $L4a1$, that is not amphicheiral, hence Table~\ref{L52} is divided into two cases. Let $ m(L4a1) $ denote the mirror image of $L4a1$. No example rises from this case.

\begin{table}[h!]
\begin{center}
\begin{tabular}{|c|c|c|c|}
\hline 
link & lift braid & equation & solutions  \\
\hline

m(L4a1) & $\Delta_{4}^{1}$ & $kp+2q-p=1$ & for all $ p \textrm{ odd}, \ L\left(p,\frac{p+1}{2} \right ), \ k=0$  \\
\hline

m(L4a1) & & & for all $ p , \ L(p,2), \ k=1$  \\[-1.5ex] &
 \raisebox{2ex}{$\Delta_{2}^{4}$} &  \raisebox{2ex}{$kp+2q-p=4$} & for all $ p \textrm{ even}, \ L\left(p,\frac{p+4}{2} \right ), \ k=0$  \\
\hline

L4a1 & $\Delta_{4}^{-1}$ & $kp+2q-p=-1$ & for all $ p \textrm{ odd}, \ L\left(p,\frac{p-1}{2} \right ), \ k=0$  \\
\hline

 L4a1 & & & for all $ p, \ L(p,p-2), \ k=-1$  \\[-1.5ex] &
 \raisebox{2ex}{$\Delta_{2}^{-4}$} &  \raisebox{2ex}{$kp+2q-p=-4$} & for all $ p \textrm{ even}, \ L\left(p,\frac{p-4}{2} \right ), \ k=0$  \\
\hline

\hline
\end{tabular}
\end{center}
\caption[legenda elenco figure]{Links in lens spaces lifting to $L4a1$ or $m(L4a1)$.}
\label{L52}
\end{table}

\subsection{Counterexamples from satellite construction}

At this stage, the examples we have found are not completely satisfactory, because it is easy to distinguish the links with equivalent lift (different number of components or different homology class).
Therefore we now construct some satellite link of the previous examples, in order to get an infinite family of different links with the same number of components and the same homology class. 

\begin{ese}{\textbf{Different links in $L(4,1)$ with cables of Hopf link as lift.}}\label{CE3}
Consider the knot $L_A$ and the link $L_B$ of Example \ref{CE}. A satellite of $L_B$ can be the link where the two patterns are described by the two braid $\tau_{n}$ and $\psi_{m}$ on $n$ and $m$ strands respectively, as in part B1) of Figure \ref{Cable}. Label $B$ such link.

\begin{figure}[h!]                      
\begin{center}                        
\includegraphics[width=14cm]{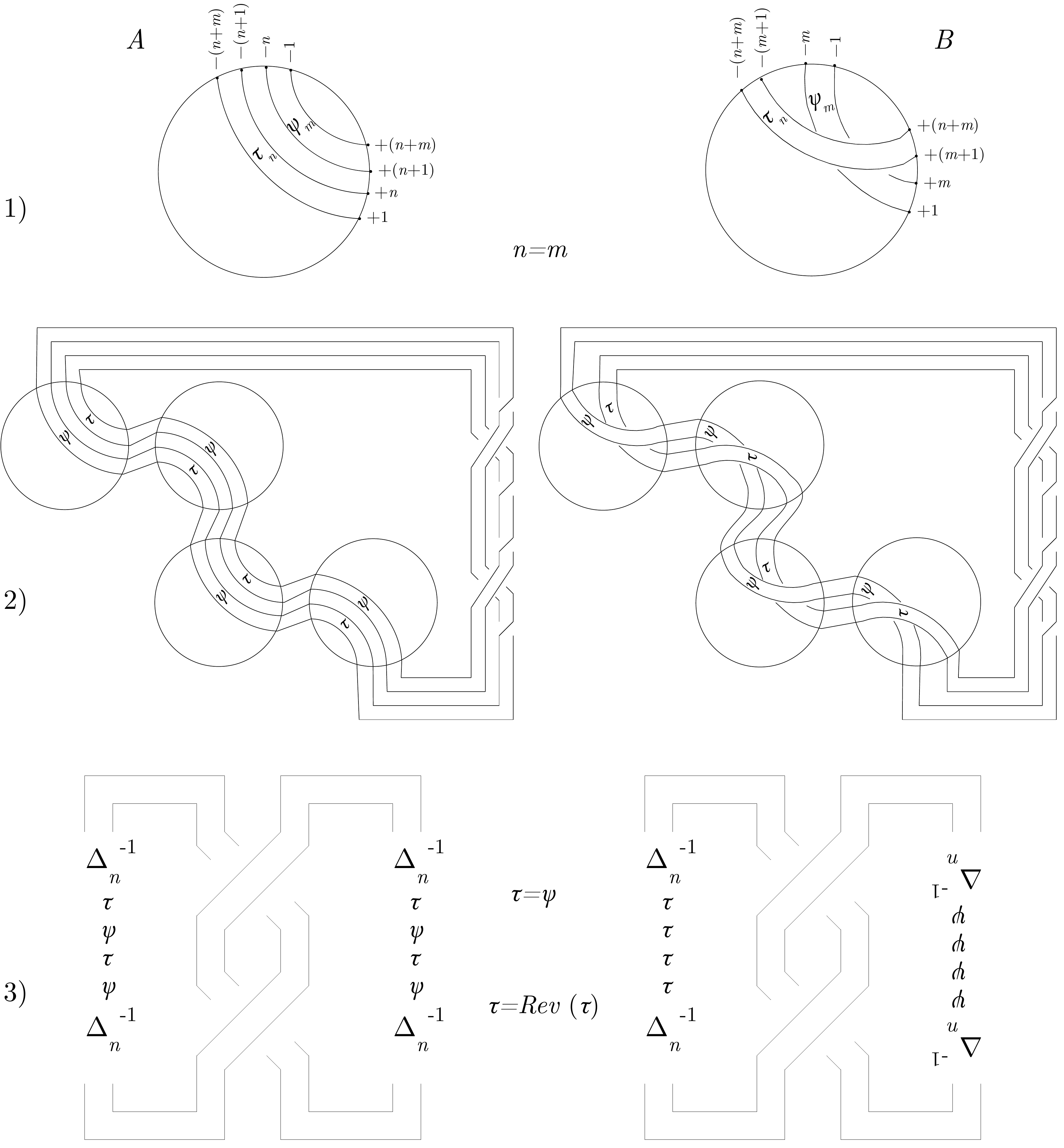}
\caption[legenda elenco figure]{Satellite construction of different links necessary to get new different links with equivalent lift.}\label{Cable}
\end{center}
\end{figure}

We need to make a satellite of the knot $L_A$ making the lift equivalent to the previous one, so we have to put the braids $\tau_{n}$ and $\psi_{m}$ on each overpass of the diagram of $L_A$, as in part A1) of Figure \ref{Cable}. Label $A$ such link.
Note that the boundary points of the two braids mix up, unless we assume $n=m$.

The lift diagram of the two considered links is illustrated in part 2) of Figure \ref{Cable} and in part 3) we make it explicit that the companion link is the Hopf link. The pattern braids are $\Delta_{n}^{-1} \tau \psi \tau \psi \Delta_{n}^{-1}$ on both sides of $A$, while for $B$ we have $\Delta_{n}^{-1} \tau^{4}  \Delta_{n}^{-1}$ and the reversed braid of $\Delta_{n}^{-1} \psi^{4}  \Delta_{n}^{-1}$.
With the assumption $\tau=\psi$ we get $\Delta_{n}^{-1} \tau^{4} \Delta_{n}^{-1}$ on both sides of $A$, whereas for $B$ we have the same braid on one side and the reversed braid on the other side.

A paper of Garside \cite{G} tells us that the operation of reversing a braid is the antihomomorphism of the braid group $Rev \colon B_{n}\rightarrow B_{n}$ which sends $\sigma_{i_{1}}\sigma_{i_{2}}\cdots \sigma_{i_{r}}$ into the braid $\sigma_{i_{r}}\sigma_{i_{r-1}}\cdots \sigma_{i_{1}}$. He proves that $Rev(\Delta)$ is equivalent to $\Delta$ into the braid group; for this reason, it is enough to assume $\tau=\textrm{Rev} (\tau)$ in order to have an equivalent lift for $A$ and $B$.
An easy example of reversible braids are palindromic ones (see \cite{DGKT} for details).

We can make some more assumptions on $\tau$ in order to handle a smaller family of links with known number of components. Let $i>0$ and $j\geq 0$ be two integer numbers and let $\tau=\Delta_{i}^{j}$, denote with $A_{i,j}$ and $B_{i,j}$ the correspondent links. The considered braid produces a pattern that is a torus link, that is to say, $A_{i,j}$ and $B_{i,j}$ are cables of $L_{A}$ and $L_{B}$. The family of these links has different behaviors for different values of $i$ and $j$:
\begin{description}

\item[for $i=1$, for all $j$:] we have $A_{1,j}=L_{A}$ and $B_{1,j}=L_{B}$;

\item[for all even $i$, for $j=0$:] the link $A_{i,0}$ and $B_{i,0}$ are equivalent (it is an easy exercise using generalized Reidemeister moves);

\item[for all odd $i$, for $j=0$:] the links $A_{i,0}$ and $B_{i,0}$ have respectively $n=i$ and $n=i+1$ components, hence they are an infinite family of different links with equivalent lift;

\item[for all odd $i>1$ or for all odd $j>0$:] the links $A_{i,j}$ and $B_{i,j}$ have a different number of components, hence they are an infinite family of different links with equivalent lift;

\item[for all even $i>1$ and for all even $j>0$:] the links $A_{i,j}$ and $B_{i,j}$ have the same number of components $n=i$, moreover each of these components has the same homology class $\delta=2$; the smaller case, $A_{2,2}$ and $B_{2,2}$ is illustrated in Figure \ref{CEAB}; we cannot prove that all the pairs of links in this family are different, anyway the below computation of the Alexander polynomials of $A_{2,2}$ and $B_{2,2}$ says that the first case consists of different links.
\end{description}

We follow \cite{CMM} for the computation of several geometric invariants of $A_{2,2}$ and $B_{2,2}$; the results are summed up in Table \ref{A3}. The letter $\nu$ denote the number of components, $A^{1}(t)$ the Alexander polynomial and $A^{-1}(t)$ the twisted Alexander polynomial. It is necessary to consider oriented links for the computation of these polynomials: we choose the orientations (shown in Table \ref{A3}) that make the corresponding oriented lifts equivalent.
 
\begin{figure}[h!]                      
\begin{center}                        
\includegraphics[width=10cm]{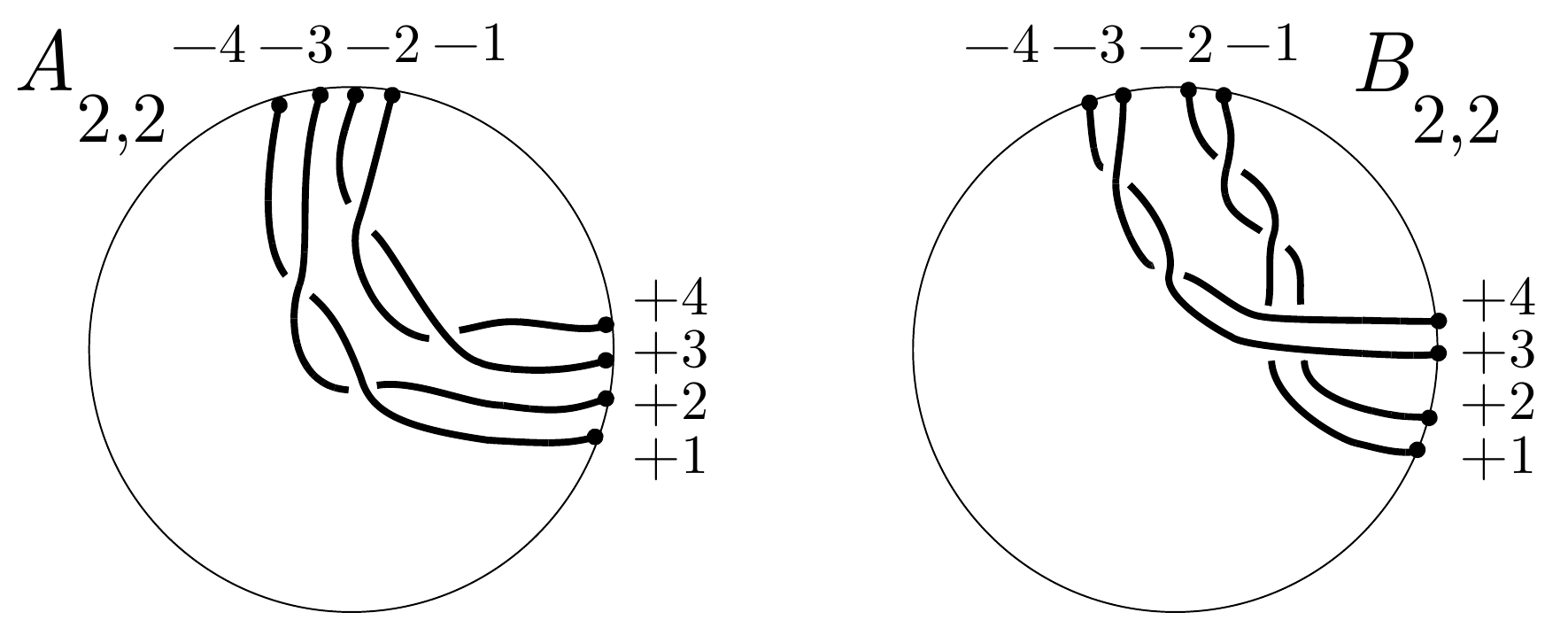}
\caption[legenda elenco figure]{Two different links with equivalent lift in $L(4,1)$.}\label{CEAB}
\end{center}
\end{figure}

\begin{table}[h!]
\begin{center}
\begin{tabular}{|c|c|c|}
\hline 
& $A_{2,2}$ & $B_{2,2}$ \\
& \includegraphics[width=3.5cm]{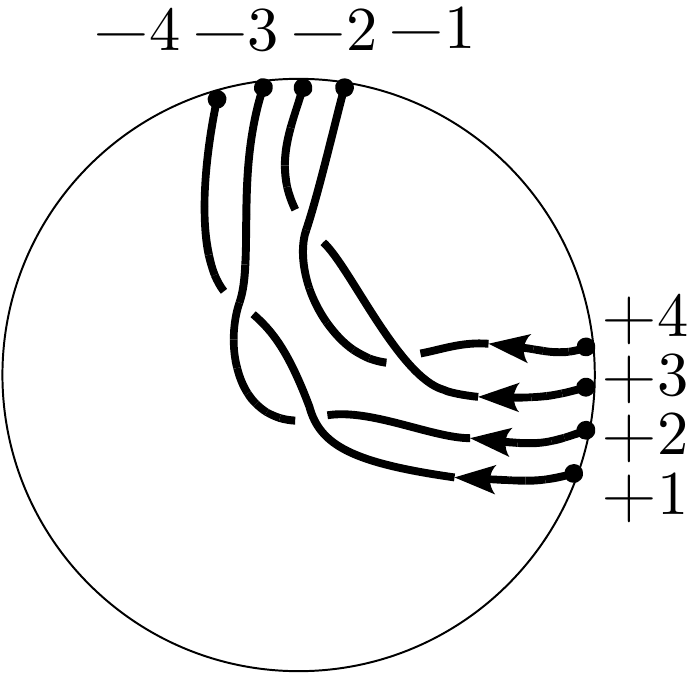} & \includegraphics[width=3.5cm]{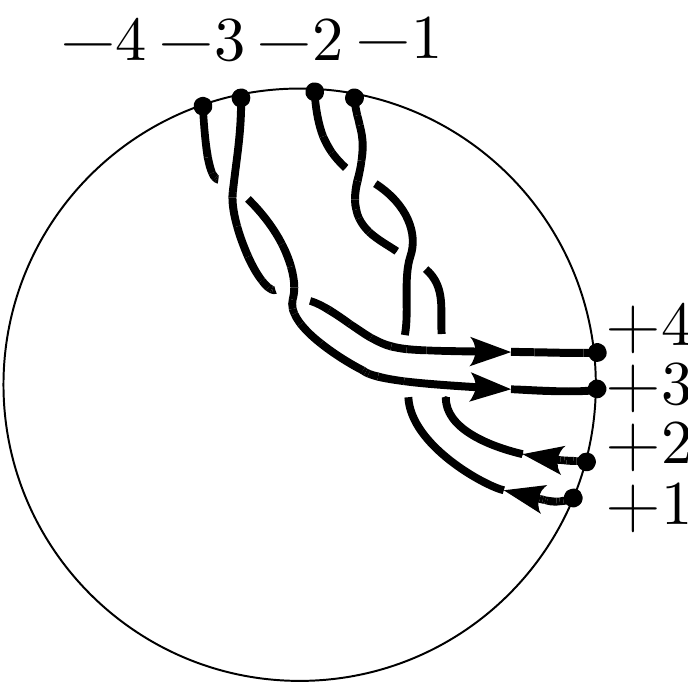} \\
\hline
$\nu$ & $2$ & $2$ \\
\hline
$[L_{i}]\subset H_{1}(L(p,q)$ & $2,2$ & $2,2$ \\
\hline
$H_{1}(L(p,q) \smallsetminus L)$ & $\mathbb{Z}\oplus \mathbb{Z}\oplus \mathbb{Z}_{2}$ & $\mathbb{Z}\oplus \mathbb{Z}\oplus \mathbb{Z}_{2}$  \\
\hline
$\bar{A}^{1}(t)$ & $t^7+t^6-t-1$ & $t^7 -t^6 +t^5 - t^4 + t^3-t^2 +t -1$ \\
$\bar{A}^{-1}(t)$ & $t^6+1 $ & $ t^6+t^4+t^2+1$ \\
 \hline
\end{tabular}
\end{center}
\caption[legenda elenco figure]{Geometric invariants of $A_{2,2}$ and $B_{2,2}$ in $L(4,1)$.}
\label{A3}
\end{table}

\end{ese}

Examples \ref{CE1}, \ref{CE} and \ref{CE3} provide different links with equivalent lift. Using this counterexamples we can produce some other infinite families of links  in the corresponding $L(p,q)$ with equivalent lift, by adding to them the same links in $\s3$ using the disjoint union and the connected sum.

Unfortunately, we have not been able to find counter-examples for all lens spaces, so we still have questions such as:
\begin{itemize}
\item is the lift a complete invariant for links in some fixed lens space, for example in the projective space?
\item is the lift a complete invariant if we restrict to primitive-homologous prime knots in $L(p,q)$ with lift different from the trivial knot?
\end{itemize}

\subsection{The case of oriented and diffeomorphic links}

Up to this stage we have considered unoriented links. Yet this lift problem can be referred also to oriented links. The answer is slightly different. Of course, we can orient the previous counter-examples and find new ones for oriented links in lens spaces.
Moreover we can consider the following property: if we take an oriented knot $K \subset L(p,q)$ such that $\widetilde K$ is invertible (i.e. it is equivalent to the knot with reversed orientations), then also the knot $-K \subset L(p,q)$ with reversed orientation has the same lift. Usually $-K$ is not equivalent to $K$ because the homology class changes. 
For links something similar happens, but you have to be careful to the orientation of each component.

Furthermore we can consider oriented links up to diffeomorphism of pairs, that is to say, two links $L_1$ and $L_2$ are equivalent in $L(p,q)$ if and only if there exists a diffeomorphism $h$ of $L(p,q)$ such that $h(L_1)=L_2$. In this case we have to examine also the following theorem of Sakuma, also proved by Boileau and Flapan about freely periodic knots.
Let $K$ be a knot in the $3$-sphere; if $\textrm{Diff}^{*}(\s3,K)$ is the group of diffeomorphisms of the pair $(\s3,K)$ which preserve the orientation of both $\s3$ and $K$, then a symmetry of a knot $K$ in $\s3$ is a finite subgroup of $\textrm{Diff}^{*}(\s3,K)$ up to conjugation.

\begin{teo}{\upshape{\cite{Sa,BF}}}
Suppose that a knot $K\subset \s3$ has free period $p$. Then there is a unique symmetry of $K$ realizing it, provided that (i) K is prime, or (ii) K is composite and the slope is specified.
\end{teo}

If we translate it into the language of knots in lens spaces, we have that the specification of the slope is equivalent to fixing the $q$ of the lens space. As a consequence, two primitive-homologous knots $K_1$ and $K_2$ in $L(p,q)$ with equivalent non-trivial lift are necessarily equivalent in $L(p,q)$. From the group of diffeotopies of $L(p,q)$ displayed in \cite{Bo} and \cite{HR}, we know that a diffeomorphism in $L(p,q)$ does not always induce an ambient isotopy of knots, so this does not provide a complete answer about the equivalence of $K_1$ and $K_2$ up to ambient isotopy.

\vspace{5mm}
\textit{Acknowledgments:} This research has been fostered during my foreign research period at Chelyabinsk State University, under the supervision of Sergey Matveev. The author is grateful to him and to all the Computational Topology and Algebra Department for hospitality and helpful discussions.
The author would also like to thank Alessia Cattabriga and Michele Mulazzani for the revision of this work.

\end{section}


\vspace{15 pt} {ENRICO MANFREDI, Department of Mathematics,
University of Bologna, ITALY. E-mail: enrico.manfredi3@unibo.it}

\end{document}